\documentclass{birkmult}
\usepackage{hyperref, url}
 \newtheorem{theorem}{Theorem}[section]
 \newtheorem{lemma}[theorem]{Lemma}
 \newtheorem{corollary}[theorem]{Corollary}

 \theoremstyle{definition}
 \newtheorem{definition}[theorem]{Definition}
 \theoremstyle{remark}
 \newtheorem{remark}[theorem]{Remark}
 
 \numberwithin{equation}{section}

\newcommand{\N}{\mathbb{N}}
\newcommand{\R}{\mathbb{R}}
\newcommand{\Z}{\mathbb{Z}}
\newcommand{\cA}{\mathcal{A}}
\newcommand{\cB}{\mathcal{B}}
\newcommand{\cM}{\mathcal{M}}
\newcommand{\cP}{\mathcal{P}}
\newcommand{\eps}{\varepsilon}

\begin{document}
%
%
%
%
%
%
%
%
%
\title[The Cauchy Singular Integral Operator]
{The Cauchy Singular Integral Operator\\
on Weighted Variable Lebesgue Spaces}
\author[A.~Yu.~Karlovich]{Alexei Yu. Karlovich}
\address{
Departamento de Matem\'atica\\
Faculdade de Ci\^encias e Tecnologia\\
Universidade Nova de Lisboa\\
Quinta da Torre\\
2829--516 Caparica\\
Portugal}
\email{oyk@fct.unl.pt}

\author[I.~M.~Spitkovsky]{Ilya M. Spitkovsky}
\address{
Department of Mathematics\\
College of William \& Mary\\
Williamsburg, VA, 23187-8795\\
U.S.A.} \email{ilya@math.wm.edu}

\subjclass{Primary 42A50; Secondary 42B25, 46E30}

\keywords{%
Weighted variable Lebesgue space,
log-H\"older continuous variable exponent,
Cauchy singular integral operator,
Hardy-Littlewood maximal operator.
}

\thanks{The first author is partially supported by FCT project PEstOE/MAT/UI4032/2011 (Portugal).}
\begin{abstract}
Let $p:\R\to(1,\infty)$ be a globally log-H\"older continuous variable exponent
and $w:\R\to[0,\infty]$ be a weight. We prove that the Cauchy singular integral
operator $S$ is bounded on the weighted variable Lebesgue space
$L^{p(\cdot)}(\R,w)=\{f:fw\in L^{p(\cdot)}(\R)\}$ if and only if the weight $w$
satisfies
\[
\sup_{-\infty<a<b<\infty}
\frac{1}{b-a}\|w\chi_{(a,b)}\|_{p(\cdot)}\|w^{-1}\chi_{(a,b)}\|_{p'(\cdot)}<\infty
\quad (1/p(x)+1/p'(x)=1).
\]
\end{abstract}
\maketitle
\section{Introduction}
Let $p:\R\to[1,\infty]$ be a measurable a.e. finite function. By
$L^{p(\cdot)}(\R)$ we denote the set of all complex-valued
functions $f$ on $\R$ such that
\[
I_{p(\cdot)}(f/\lambda):=\int_{\R} |f(x)/\lambda|^{p(x)} dx <\infty
\]
for some $\lambda>0$. This set becomes a Banach space when
equipped with the norm
\[
\|f\|_{p(\cdot)}:=\inf\big\{\lambda>0: I_{p(\cdot)}(f/\lambda)\le 1\big\}.
\]
It is easy to see that if $p$ is constant, then $L^{p(\cdot)}(\R)$ is nothing but
the standard Lebesgue space $L^p(\R)$. The space $L^{p(\cdot)}(\R)$
is referred to as a \textit{variable Lebesgue space}.

A measurable function $w:\R\to[0,\infty]$ is referred to as a \textit{weight}
whenever $0<w(x)<\infty$ a.e. on $\R$. Given a variable exponent $p:\R\to[1,\infty]$ and
a weight $w:\R\to[0,\infty]$, we define  the weighted variable exponent space
$L^{p(\cdot)}(\R,w)$ as the space of all measurable complex-valued functions $f$
such that $fw\in L^{p(\cdot)}(\R)$. The norm on this space is naturally defined
by
\[
\|f\|_{p(\cdot),w}:=\|fw\|_{p(\cdot)}.
\]

Given $f\in L_{\rm loc}^1(\R)$, the Hardy-Littlewood maximal operator is
defined by
\[
Mf(x):=\sup_{Q\ni x}\frac{1}{|Q|}\int_Q|f(y)|dy
\]
where the supremum is taken over all intervals $Q\subset\R$
containing $x$. The Cauchy singular integral operator $S$ is defined for
$f\in L^1_{\rm loc}(\R)$ by
\[
(Sf)(x):=\frac{1}{\pi i}\int_\R\frac{f(\tau)}{\tau-x}d\tau
\quad(x\in\R),
\]
where the integral is understood in the principal value sense.

Following \cite[Section~2]{CDH11} or \cite[Section~4.1]{DHHR11}, one says that
$\alpha:\R\to\R$ is locally $\log$-H\"older continuous if there exists $c_1>0$
such that
\[
|\alpha(x)-\alpha(y)|\le \frac{c_1}{\log(e+1/|x-y|)}
\]
for all $x,y\in\R$. Further, $\alpha$ is said to satisfy the $\log$-H\"older decay
condition if there exist $\alpha_\infty\in\R$ and a constant $c_2>0$ such that
\[
|\alpha(x)-\alpha_\infty|\le\frac{c_2}{\log(e+|x|)}
\]
for all $x\in\R$. One says that $\alpha$ is globally $\log$-H\"older continuous
on $\R$ if it is locally $\log$-H\"older continuous and satisfies the $\log$-H\"older
decay condition. Put
\[
p_-:=\operatornamewithlimits{ess\,inf}_{x\in\R}p(x),
\quad
\operatornamewithlimits{ess\,sup}_{x\in\R}p(x)=:p_+.
\]

As usual, we use the convention $1/\infty:=0$ and denote by $\cP^{\log}(\R)$
the set of all variable exponents such that $1/p$ is globally $\log$-H\"older
continuous. If $p\in\cP^{\log}(\R)$, then the limit
\[
\frac{1}{p(\infty)}:=\lim_{|x|\to\infty}\frac{1}{p(x)}
\]
exists. If $p_+<\infty$, then $p\in\cP^{\log}(\R)$ if and only if $p$ is globally
$\log$-H\"older continuous.

By \cite[Theorem~4.3.8]{DHHR11}, if $p\in\cP^{\log}(\R)$ with $p_->1$,
then the Hardy-Littlewood maximal operator $M$ is bounded on $L^{p(\cdot)}(\R)$.
Notice, however, that the condition $p\in\cP^{\log}(\R)$ is not necessary,
there are even discontinuous exponents $p$ such that $M$ is bounded on
$L^{p(\cdot)}(\R)$. Corresponding examples were first constructed by Lerner
and they are contained in \cite[Section~5.1]{DHHR11}.

In this paper we will mainly suppose that
\begin{equation}\label{eq:exponents}
1<p_-,\quad p_+<\infty.
\end{equation}
Under these conditions, the space $L^{p(\cdot)}(\R)$ is separable and reflexive,
and its Banach space dual $[L^{p(\cdot)}(\R)]^*$ is isomorphic to $L^{p'(\cdot)}(\R)$, where
\[
1/p(x)+1/p'(x)=1 \quad(x\in\R)
\]
(see \cite[Chap.~3]{DHHR11}). If, in addition, $w\chi_E\in L^{p(\cdot)}(\R)$ and
$\chi_E/w\in L^{p'(\cdot)}(\R)$ for any measurable set $E\subset\R$ of finite measure,
then $L^{p(\cdot)}(\R,w)$ is a Banach function space and $[L^{p(\cdot)}(\R,w)]^*=L^{p(\cdot)}(\R,w^{-1})$.
Here and in what follows, $\chi_E$ denotes the characteristic function of the set $E$.

Probably, one of the simplest weights is the following power weight
\begin{equation}\label{eq:power}
w(x):=|x-i|^{\lambda_\infty}\prod_{j=1}^m|x-x_j|^{\lambda_j},
\end{equation}
where $-\infty<x_1<\dots<x_m<+\infty$ and $\lambda_1,\dots,\lambda_m,\lambda_\infty\in\R$.
Kokilashvili, Paatashvili, and Samko studied the boundedness
of the operators $M$ and $S$ on $L^{p(\cdot)}(\R,w)$ with power weights
\eqref{eq:power}. From \cite[Theorem~A]{KPS06} and \cite[Theorem~B]{KS08} one
can extract the following result.
\begin{theorem}\label{th:KS}
Let $p\in\cP^{\log}(\R)$ satisfy \eqref{eq:exponents} and $w$ be a power weight
\eqref{eq:power}.

\begin{enumerate}
\item[{\rm(a)}] \textbf{(Kokilashvili, Samko).}
Suppose, in addition, that $p$ is constant outside an interval containing
$x_1,\dots,x_m$. Then the Hardy-Littlewood maximal operator
$M$ is bounded on $L^{p(\cdot)}(\R,w)$ if and only if
\begin{equation}\label{eq:KS}
0<\frac{1}{p(x_j)}+\lambda_j<1\text{ for } j\in\{1,\dots,m\},
\quad
0<\frac{1}{p(\infty)}+\lambda_\infty+\sum_{j=1}^m\lambda_j<1.
\end{equation}

\item[{\rm(b)}] \textbf{(Kokilashvili, Paatashvili, Samko).}
The Cauchy singular integral operator $S$ is bounded on $L^{p(\cdot)}(\R,w)$
if and only if \eqref{eq:KS} is fulfilled.
\end{enumerate}
\end{theorem}
Further, the sufficiency portion of this result was extended in \cite{KSS07-maximal,KSS07-singular}
to radial oscillating weights of the form $\prod_{j=1}^m\omega_j(|x-x_j|)$, where
$\omega_j(t)$ are continuous functions for $t>0$ that may oscillate near zero and whose
Matuszewska-Orlicz indices can be different. Notice that the Matuszewska-Orlicz
indices of $\omega_j(t)=t^{\lambda_j}$ are both equal to $\lambda_j$.

Very recently, Cruz-Uribe, Diening, and H\"ast\"o \cite[Theorem~1.3]{DHHR11}
generalized part (a) of Theorem~\ref{th:KS} to the case of general weights.
To formulate their result, we will introduce the following generalization of the classical
Muckenhoupt condition (written in the symmetric form). We say that a weight
$w:\R\to[0,\infty]$ belongs to the class $\cA_{p(\cdot)}(\R)$ if
\[
\sup_{-\infty<a<b<\infty}
\frac{1}{b-a}\|w\chi_{(a,b)}\|_{p(\cdot)}\|w^{-1}\chi_{(a,b)}\|_{p'(\cdot)}<\infty.
\]
This condition goes back to Berezhnoi \cite{B99} (in the more general setting
of Banach function spaces), it was studied by the first author \cite{K03}
(in the case of Banach function spaces defined on Carleson curves)
and Kopaliani \cite{K07}.
\begin{theorem}[Cruz-Uribe, Diening, H\"ast\"o]
\label{th:CDH}
Let $p\in\cP^{\log}(\R)$ satisfy \eqref{eq:exponents} and
$w:\R\to[0,\infty]$ be a weight. The Hardy-Littlewood maximal
operator $M$ is bounded on the weighted variable Lebesgue space
$L^{p(\cdot)}(\R,w)$ if and only if $w\in\cA_{p(\cdot)}(\R)$.
\end{theorem}
The aim of this paper is to generalize part (b) of Theorem~\ref{th:KS} to the
case of general weights. We will prove the following.
\begin{theorem}[Main result]
\label{th:main}
Let $p\in\cP^{\log}(\R)$ satisfy \eqref{eq:exponents} and
$w:\R\to[0,\infty]$ be a weight. The Cauchy singular integral
operator $S$ is bounded on the weighted variable Lebesgue space
$L^{p(\cdot)}(\R,w)$  if and only if $w\in\cA_{p(\cdot)}(\R)$.
\end{theorem}
From this theorem, by using standard techniques, we derive also the following.
\begin{theorem}
\label{th:nice-relations}
Let $p\in\cP^{\log}(\R)$ satisfy \eqref{eq:exponents} and $w\in\cA_{p(\cdot)}(\R)$.
Then $S^2=I$ on the space $L^{p(\cdot)}(\R,w)$ and $S^*=S$ on the space $L^{p'(\cdot)}(\R,w^{-1})$.
\end{theorem}
The paper is organized as follows. In Section~\ref{sec:wbfs} we collect
necessary facts on Banach function spaces $X(\R)$ in the sense of Luxemburg
and discuss weighted Banach functions spaces $X(\R,w)=\{f:fw\in X(\R)\}$.
A special attention is paid to conditions implying that $X(\R,w)$
is a Banach function space itself, to separability and reflexivity of $X(\R,w)$,
and to density of smooth compactly supported functions in $X(\R,w)$ and in its
dual space $X'(\R,w^{-1})$. In Section~\ref{subsec:sharp} we prepare the proof
of a sufficient condition for the boundedness of the operator $S$ and formulate
two key estimates by Lerner \cite{L04} and \'Alvarez and P\'erez \cite{AP94}.
On the basis of these results, in Section~\ref{subsec:sufficiency} we prove that
if $X(\R)$ is a separable Banach function space and the Hardy-Littlewood
maximal function is bounded on the weighted Banach function spaces
$X(\R,w)$ and $X'(\R,w^{-1})$, then $S$ is bounded on $X(\R,w)$ and $S^2=I$.
Moreover, if $X(\R)$ is reflexive, then $S^*$ coincides with $S$ on $X'(\R,w^{-1})$.
In Section~\ref{subsec:necessity} we prove that if $S$ is bounded on
the weighted Banach function spaces $X(\R,w)$, then
\[
\sup_{-\infty<a<b<\infty}
\frac{1}{b-a}\|w\chi_{(a,b)}\|_{X(\R)}\|w^{-1}\chi_{(a,b)}\|_{X'(\R)}<\infty
\]
where $X'(\R)$ is the associate space for $X(\R)$. Finally, in Section~\ref{subsec:WVLS}
we explain that Theorems~\ref{th:main} and \ref{th:nice-relations} follow from
results of Sections~\ref{subsec:sufficiency}--\ref{subsec:necessity} and
Theorem~\ref{th:CDH} because $L^{p(\cdot)}(\R)$ is a Banach function space,
which is separable and reflexive whenever $p$ satisfies \eqref{eq:exponents}.
\section{Weighted Banach function spaces}\label{sec:wbfs}
\subsection{Banach function spaces}
The set of all Lebesgue measurable complex-valued functions on $\R$ is
denoted by $\cM$. Let $\cM^+$ be the subset of functions in $\cM$ whose
values lie  in $[0,\infty]$. The characteristic function of a measurable
set $E\subset\R$ is denoted by $\chi_E$ and the Lebesgue measure of $E$
is denoted by $|E|$.
\begin{definition}[{\cite[Chap.~1, Definition~1.1]{BS88}}]
\label{def-BFS}
A mapping $\rho:\cM^+\to [0,\infty]$ is
called a {\it Banach function norm} if, for all functions $f,g,
f_n \ (n\in\N)$ in $\cM^+$, for all constants $a\ge 0$, and for
all measurable subsets $E$ of $\R$, the following properties
hold:
\begin{eqnarray*}
{\rm (A1)} & & \rho(f)=0  \Leftrightarrow  f=0\ \mbox{a.e.}, \quad
\rho(af)=a\rho(f), \quad
\rho(f+g) \le \rho(f)+\rho(g),\\
{\rm (A2)} & &0\le g \le f \ \mbox{a.e.} \ \Rightarrow \ \rho(g)
\le \rho(f)
\quad\mbox{(the lattice property)},\\
{\rm (A3)} & &0\le f_n \uparrow f \ \mbox{a.e.} \ \Rightarrow \
       \rho(f_n) \uparrow \rho(f)\quad\mbox{(the Fatou property)},\\
{\rm (A4)} & & |E|<\infty \Rightarrow \rho(\chi_E) <\infty,\\
{\rm (A5)} & & |E|<\infty \Rightarrow \int_E f(x)\,dx \le C_E\rho(f)
\end{eqnarray*}
with $C_E \in (0,\infty)$ which may depend on $E$ and $\rho$ but is
independent of $f$.
\end{definition}
When functions differing only on a set of measure zero are identified,
the set $X(\R)$ of all functions $f\in\cM$ for which $\rho(|f|)<\infty$ is
called a \textit{Banach function space}. For each $f\in X(\R)$, the norm of
$f$ is defined by
\[
\|f\|_{X(\R)} :=\rho(|f|).
\]
The set $X(\R)$ under the natural linear space operations and under this norm
becomes a Banach space (see \cite[Chap.~1, Theorems~1.4 and~1.6]{BS88}).

If $\rho$ is a Banach function norm, its associate norm $\rho'$ is
defined on $\cM^+$ by
\[
\rho'(g):=\sup\left\{
\int_\R f(x)g(x)\,dx \ : \ f\in \cM^+, \ \rho(f) \le 1
\right\}, \quad g\in \cM^+.
\]
It is a Banach function norm itself \cite[Chap.~1, Theorem~2.2]{BS88}.
The Banach function space $X'(\R)$ determined by the Banach function norm
$\rho'$ is called the \textit{associate space} (\textit{K\"othe dual}) of $X(\R)$.
The associate space $X'(\R)$ is a subspace of the dual space $[X(\R)]^*$.
The construction of the associate space implies the following
H\"older inequality for Banach function spaces.
\begin{lemma}[{\cite[Chap.~1, Theorem~2.4]{BS88}}]
\label{le:Hoelder}
Let $X(\R)$ be a Banach function  space and $X'(\R)$ be its associate space.
If $f\in X(\R)$ and $g\in X'(\R)$, then $fg$ is integrable and
\[
\|fg\|_{L^1(\R)}\le \|f\|_{X(\R)}\|g\|_{X'(\R)}.
\]
\end{lemma}
The next result provides a useful converse to the integrability assertion
of Lemma~\ref{le:Hoelder}.
\begin{lemma}[{\cite[Chap.~1, Lemma~2.6]{BS88}}]
\label{le:Landau}
Let $X(\R)$ be a Banach function space. In order that a measurable function
$g$ belong to the associate space $X'(\R)$, it is necessary and sufficient
that $fg$ be integrable for every $f$ in $X(\R)$.
\end{lemma}
\subsection{Weighted Banach function spaces}
Let $X(\R)$ be a Banach function space generated by a Banach function norm $\rho$.
We say that $f\in X_{\rm loc}(\R)$ if $f\chi_E\in X(\R)$ for any measurable
set $E\subset\R$ of finite measure.
A measurable function $w:\R\to[0,\infty]$ is referred to as a \textit{weight}
if $0<w(x)<\infty$ a.e. on $\R$. Define the mapping
$\rho_w:\cM^+\to[0,\infty]$ and the set $X(\R,w)$ by
\[
\rho_w(f):=\rho(fw)
\quad (f\in\cM^+),
\quad\quad
X(\R,w):=\big\{f\in\cM^+:\ fw\in X(\R)\big\}.
\]
\begin{lemma}\label{le:WBFS}
Let $X(\R)$ be a Banach function space generated by a Banach function norm $\rho$,
let $X'(\R)$ be its associate space, and let $w:\R\to[0,\infty]$ be a weight.
\begin{enumerate}
\item[{\rm(a)}]
The mapping $\rho_w$ satisfies Axioms {\rm(A1)-(A3)} in Definition~{\rm\ref{def-BFS}}
and $X(\R,w)$ is a linear normed space with respect to the norm
\[
\|f\|_{X(\R,w)}:=\rho_w(|f|)=\rho(|fw|)=\|fw\|_{X(\R)}.
\]

\item[{\rm (b)}]
If $w\in X_{\rm loc}(\R)$ and $1/w\in X_{\rm loc}'(\R)$, then $\rho_w$
is a Banach function norm and $X(\R,w)$ is a Banach function space generated
by $\rho_w$.

\item[{\rm (c)}]
If $w\in X_{\rm loc}(\R)$ and $1/w\in X_{\rm loc}'(\R)$, then
$X'(\R,w^{-1})$ is the associate space for the Banach function space
$X(\R,w)$.
\end{enumerate}
\end{lemma}
\begin{proof}
The proof is analogous to that one of \cite[Lemma~2.5]{K03}.

Part (a) follows from Axioms (A1)-(A3) for the Banach function norm $\rho$
and the fact that $0<w(x)<\infty$ almost everywhere on $\R$.

(b) If $w\in X_{\rm loc}(\R)$, then $w\chi_E\in X(\R)$ for every
measurable set $E\subset\R$ of finite measure. Therefore $\rho_w(\chi_E)=\rho(w\chi_E)<\infty$.
Then $\rho_w$ satisfies Axiom (A4).

Since $1/w\in X_{\rm loc}'(\R)$, we have $C_E:=\rho'(\chi_E/w)<\infty$ for
every measurable set $E\subset\R$ of finite measure. On the other hand, by
Axiom (A2), for $f\in\cM^+$ we have $\rho(fw\chi_E)\le\rho(fw)=\rho_w(f)$.
By H\"older's inequality for $\rho$ (Lemma~\ref{le:Hoelder}), we obtain
\[
\int_E f(x)\,dx=\int_\R f(x)w(x)\chi_E(x) \cdot \frac{\chi_E(x)}{w(x)}\,dx
\le\rho(fw\chi_E)\rho'(\chi_E/w)\le C_E\rho_w(f).
\]
Thus $\rho_w$ satisfies Axiom (A5), that is, $X(\R,w)$ is a Banach
function space. Part (b) is proved.

(c) For $g\in\cM^+$, we have
\[
\begin{split}
(\rho_w)'(g) &=\sup\left\{\int_\R f(x)g(x)\,dx:\ f\in\cM^+,\ \rho_w(f)\le 1\right\}
\\
&=
\sup\left\{\int_\R\big(f(x)w(x)\big)\left(\frac{g(x)}{w(x)}\right)dx: \ f\in\cM^+,\ \rho(fw)\le 1\right\}
\\
&=
\sup\left\{\int_\R h(x)\left(\frac{g(x)}{w(x)}\right)dx: \ h\in\cM^+,\ \rho(h)\le 1\right\}
\\
&=\rho'(g/w).
\end{split}
\]
Hence $(X(\R,w))'=X'(\R,w^{-1})$.
\end{proof}
From Lemma~\ref{le:WBFS} and the Lorentz-Luxemburg theorem (see e.g.
\cite[Chap.~1, Theorem~2.7]{BS88}) we obtain the following.
\begin{lemma}
Let $X(\R)$ be a Banach function space and $w:\R\to[0,\infty]$ be a weight
such that $w\in X_{\rm loc}(\R)$ and $1/w\in X_{\rm loc}'(\R)$. Then
\begin{equation}\label{eq:corollary-LS-1}
\|f\|_{X(\R,w)}
=
\sup\left\{
\int_\R |f(x)g(x)|\,dx \ : \ g\in X'(\R,w^{-1}), \ \|g\|_{X'(\R,w^{-1})} \le 1
\right\}
\end{equation}
for all $f\in X(\R,w)$ and
\begin{equation}\label{eq:corollary-LS-2}
\|g\|_{X'(\R,w^{-1})}
=
\sup\left\{
\int_\R |f(x)g(x)|\,dx \ : \ f\in X(\R,w), \ \|f\|_{X(\R,w)} \le 1
\right\}
\end{equation}
for all $g\in X'(\R,w^{-1})$.
\end{lemma}
\subsection{Reflexivity of weighted Banach function spaces}
A function $f$ in a Banach function space $X(\R)$ is said to have
\textit{absolutely continuous norm} in $X(\R)$ if $\|f\chi_{E_n}\|_{X(\R)}\to 0$
for every sequence $\{E_n\}_{n=1}^\infty$ of measurable sets on $\R$
satisfying $\chi_{E_n}\to 0$ a.e. on $\R$ as $n\to\infty$. If all
functions $f\in X(\R)$ have this property, then the space $X(\R)$ itself
is said to have \textit{absolutely continuous norm} (see
\cite[Chap.~1, Section~3]{BS88}).
\begin{lemma}[{\cite[Chap.~1, Lemma~3.4]{BS88}}]
\label{le:AC-property}
Let $X(\R)$ be a Banach function space. If $f\in X(\R)$ has absolutely continuous
norm, then to each $\eps>0$ there corresponds $\delta>0$ such that $|E|<\delta$
implies $\|f\chi_E\|_{X(\R)}<\eps$.
\end{lemma}
\begin{lemma}\label{le:AC}
Let $X(\R)$ be a Banach function space and $w:\R\to[0,\infty]$ be a weight
such that $w\in X_{\rm loc}(\R)$ and $1/w\in X_{\rm loc}'(\R)$.
If $X(\R)$ has absolutely continuous norm, then $X(\R,w)$ has absolutely
continuous norm too.
\end{lemma}
\begin{proof}
The proof is a literal repetition of that one of \cite[Proposition~2.6]{K03}.
By Lemma~\ref{le:WBFS}(b), $X(\R,w)$ is a Banach function space.
If $f\in X(\R,w)$, then $fw\in X(\R)$ has absolutely continuous norm in $X(\R)$.
Therefore,
\[
\|f\chi_{E_n}\|_{X(\R,w)}=\|fw\chi_{E_n}\|_{X(\R)}\to 0
\]
for every
sequence $\{E_n\}_{n=1}^\infty$ of measurable sets on $\R$ satisfying
$\chi_{E_n}\to 0$ a.e. on $\R$ as $n\to \infty$. Thus, $f\in X(\R,w)$
has absolutely continuous norm in $X(\R,w)$.
\end{proof}
From Lemma~\ref{le:WBFS} and \cite[Chap.~1, Corollaries 4.3, 4.4]{BS88}
we obtain the following.
\begin{lemma}\label{le:duality-reflexivity}
Let $X(\R)$ be a Banach function space and $w:\R\to[0,\infty]$ be a weight
such that $w\in X_{\rm loc}(\R)$ and $1/w\in X_{\rm loc}'(\R)$.

\begin{enumerate}
\item[{\rm (a)}]
The Banach space dual $[X(\R,w)]^*$ of the weighted Banach function space
$X(\R,w)$ is isometrically isomorphic to the associate space
$X'(\R,w^{-1})$ if and only if $X(\R,w)$ has absolutely continuous norm.
If this is the case, then the general form of a linear
functional on $X(\R,w)$ is given by
\[
G(f):=\int_\R f(x)\overline{g(x)}\,dx
\quad\text{for}\quad
g\in X'(\R,w^{-1})
\]
and $\|G\|_{[X(\R,w)]^*}=\|g\|_{X'(\R,w^{-1})}$.

\item[{\rm (b)}]
The weighted Banach function space $X(\R,w)$ is reflexive if and only if
both $X(\R,w)$ and $X'(\R,w^{-1})$ have absolutely continuous norm.
\end{enumerate}
\end{lemma}
\begin{corollary}\label{co:weighted-reflexivity}
Let $X(\R)$ be a Banach function space and $w:\R\to[0,\infty]$ be a weight
such that $w\in X_{\rm loc}(\R)$ and $1/w\in X_{\rm loc}'(\R)$.
If $X(\R)$ is reflexive, then $X(\R,w)$ is reflexive.
\end{corollary}
\begin{proof}
The proof is a literal repetition of that one of \cite[Corollary~2.8]{K03}.
If $X(\R)$ is reflexive, then, by \cite[Chap.~1, Corollary~4.4]{BS88},
both $X(\R)$ and $X'(\R)$ have absolutely continuous norm. In that case, due to
Lemma~\ref{le:AC}, both $X(\R,w)$ and $X'(\R,w^{-1})$ have absolutely
continuous norm. By Lemma~\ref{le:duality-reflexivity}(b), $X(\R,w)$ is reflexive.
\end{proof}
\subsection{Density of smooth compactly supported functions}
For a subset $Y$ of $L^\infty(\R)$, let $Y_0$ denote the set of all compactly
supported functions in $Y$.
\begin{lemma}\label{le:nice-subset}
Let $X(\R)$ be a Banach function space.
\begin{enumerate}
\item[{\rm(a)}]
$L_0^\infty(\R)\subset X(\R)$.

\item[{\rm(b)}]
If $X(\R)$ has absolutely continuous norm, then $L_0^\infty(\R)$,
$C_0(\R)$, and $C_0^\infty(\R)$ are dense in $X(\R)$.
\end{enumerate}
\end{lemma}
\begin{proof}
Part (a) follows from the definition of a Banach function space.

(b) From \cite[Chap.~1, Proposition~3.10 and Theorem~3.11]{BS88} it follows that
$L_ 0^\infty(\R)$ is dense in $X(\R)$.

Let us show that each function $u\in L_0^\infty(\R)$ can be approximated by a
function from $C_0(\R)$ in the norm of $X(\R)$. We have ${\rm supp}\,u\subset Q$
and $|u(x)|\le a$ for almost all $x\in\R$, where $Q$ is some finite closed segment
and $a>0$. By Axiom (A4), $\chi_Q\in X(\R)$ and $\chi_Q$ has absolutely continuous
norm by the hypothesis. From Lemma~\ref{le:AC-property} it follows that for
every $\eps>0$ there is a $\delta>0$ such that $|E|<\delta$ implies that
$\|\chi_Q\chi_E\|_{X(\R)}<\eps$. By Luzin's theorem, for such a $\delta>0$ there
is a continuous function $v$ supported in $Q$ such that $|v(x)|\le a$ and the
measure of the set $\widetilde{Q}:=\{x\in Q:u(x)\ne v(x)\}$ is less than $\delta$.
Then
\[
|u(x)-v(x)|\le 2a\chi_{\widetilde{Q}}(x)\quad (x\in\R).
\]
Therefore, by Axiom (A2),
\[
\|u-v\|_{X(\R)}\le 2a\|\chi_Q\chi_{\widetilde{Q}}\|_{X(\R)}<2a\eps.
\]
Hence, each function $u\in L_0^\infty(\R)$ can be approximated by a function
from $C_0(\R)$ in the norm of $X(\R)$. Thus, $C_0(\R)$ is dense in $X(\R)$.

Now let us prove that each function $v\in C_0(\R)$ can be approximated by a
function from $C_0^\infty(\R)$ in the norm of $X(\R)$. Let $a\in C_0^\infty(\R)$
and $\int_\R a(x)dx=1$. Consider
\[
v_t(x)=\frac{1}{t}\int_\R a\left(\frac{y}{t}\right)v(x-y)\,dy\quad (t>0).
\]
It is easy to see that $v_t\in C_0^\infty(\R)$. Fix an interval $Q$
containing the supports of $v$ and $v_t$. Then for every $\eps>0$ there is
a $t>0$ such that $|v_t(x)-v(x)|<\eps$ for all $x\in Q$. Hence,
\[
\|v_t-v\|_{X(\R)}=\|(v_t-v)\chi_Q\|_{X(\R)}<\eps\|\chi_Q\|_{X(\R)},
\]
that is, $v\in C_0(\R)$ can be approximated by a function from $C_0^\infty(\R)$
in the norm of $X(\R)$. Thus, $C_0^\infty(\R)$ is dense in $X(\R)$.
\end{proof}
From \cite[Chap.~1, Corollary~5.6]{BS88} one can extract the following.
\begin{lemma}\label{le:separability-AC}
A Banach function space $X(\R)$ is separable if and only if it has absolutely
continuous norm.
\end{lemma}
Gathering the results mentioned above, we arrive at the next result.
\begin{lemma}\label{le:density}
Let $X(\R)$ be a Banach function space and $w:\R\to[0,\infty]$ be a weight
such that $w\in X_{\rm loc}(\R)$ and $1/w\in X_{\rm loc}'(\R)$.
\begin{enumerate}
\item[{\rm (a)}]
If $X(\R)$ is separable, then $L_0^\infty(\R)$, $C_0(\R)$, and $C_0^\infty(\R)$ are
dense in the weighted Banach function space $X(\R,w)$.

\item[{\rm (b)}]
If $X(\R)$ is reflexive, then $L_0^\infty(\R)$, $C_0(\R)$, and $C_0^\infty(\R)$ are
dense in the weighted Banach function spaces $X(\R,w)$ and $X'(\R,w^{-1})$.
\end{enumerate}
\end{lemma}
\begin{proof}
(a) If $X(\R)$ is separable, then by Lemma~\ref{le:separability-AC}, $X(\R)$
has absolutely continuous norm. Therefore $X(\R,w)$ has absolutely
continuous norm too, in view of Lemma~\ref{le:AC}. Hence, from
Lemma~\ref{le:nice-subset}(b) we derive that $L_0^\infty(\R)$,
$C_0(\R)$, and $C_0^\infty(\R)$ are dense in $X(\R,w)$. Part (a) is proved.

\medskip
(b) If $X(\R)$ is reflexive, then by \cite[Chap.~1, Corollary~4.4]{BS88},
both $X(\R)$ and $X'(\R)$ have absolutely continuous norm. Hence both
$X(\R,w)$ and $X'(\R,w^{-1})$ have absolutely continuous norm in view of
Lemma~\ref{le:AC}. Thus, from Lemma~\ref{le:nice-subset}(b) we get that
$L_0^\infty(\R)$, $C_0(\R)$, and $C_0^\infty(\R)$ are dense in $X(\R,w)$
and in $X'(\R,w^{-1})$.
\end{proof}
\section{Boundedness of the Cauchy singular integral operator on weighted Banach function spaces}
\subsection{Well-known properties of the Cauchy singular integral operator}
One says that a linear operator $T$ from $L^1(\R)$ into the space of
complex-valued measurable functions on $\R$ is of weak type $(1,1)$
if for every $\alpha>0$,
\[
|\{x\in\R:|(Tf)(x)|>\alpha\}|\le\frac{C_K}{\alpha}\|f\|_{L^1(\R)}
\]
with some absolute constant $C_K>0$.

The following results are proved in many standard  texts on Harmonic Analysis,
see e.g. \cite[Chap.~3, Theorem 4.9(b)]{BS88} or \cite[pp.~51--52]{D01}.
\begin{theorem}\label{th:Kolmogorov-Riesz}
\begin{enumerate}
\item[{\rm(a)}] \textbf{(Kolmogorov).}
The Cauchy singular integral operator $S$ is of weak type $(1,1)$.

\item[{\rm(b)}] \textbf{(M.~Riesz).}
The Cauchy singular integral operator $S$ is bounded on $L^p(\R)$ for every
$p\in(1,\infty)$.
\end{enumerate}
\end{theorem}
\begin{theorem}\label{th:S-L2}
If $f,g\in L^2(\R)$, then
\begin{eqnarray}
(S^2f)(x) &=& f(x)\quad (x\in\R),
\label{eq:nice-relations-1}
\\
\int_\R (Sf)(x)\overline{g(x)}\,dx &=& \int_\R f(x)\overline{(Sg)(x)}\,dx.
\label{eq:nice-relations-2}
\end{eqnarray}
\end{theorem}
\subsection{Pointwise estimates for sharp maximal functions}\label{subsec:sharp}
For $\delta>0$ and $f\in L^\delta_{\rm loc}(\R)$, set
\[
f_{\delta}^{\#}(x):=\sup_{Q\ni x}\inf_{c\in\R}
\left(\frac{1}{|Q|}\int_Q|f(y)-c|^{\delta}dy\right)^{1/\delta}.
\]
The non-increasing rearrangement (see, e.g., \cite[Chap.~2,
Section~1]{BS88}) of a measurable function $f$ on $\R$ is
defined by
\[
f^*(t):=
\inf\big\{\lambda>0:|\{x\in\R:|f(x)|>\lambda\}|\le t\big\}
\quad(0<t<\infty).
\]
For a fixed $\lambda\in(0,1)$ and a given measurable function $f$
on $\R$, consider the local sharp maximal function
$M^{\#}_{\lambda}f$ defined by
\[
M^{\#}_{\lambda}f(x) := \sup_{Q\ni x}\inf_{c\in\R}
\big((f-c)\chi_Q\big)^*\left(\lambda|Q|\right).
\]
In all above definitions the suprema are taken over all intervals
$Q\subset\R$ containing $x$.

The following result was proved by Lerner \cite[Theorem~1]{L04}
for the case of $\R^n$.
\begin{theorem}[Lerner]\label{th:Lerner}
For a function $g\in L_{\rm loc}^1(\R)$ and a measurable function $\varphi$
satisfying
\begin{equation}\label{eq:Lerner}
|\{x\in\R:|\varphi(x)|>\alpha\}|<\infty \quad\mbox{for all}\quad
\alpha>0,
\end{equation}
one has
\[
\int_\R|\varphi(x)g(x)|\,dx
\le
C_L\int_\R M_{\lambda}^\# \varphi(x)Mg(x)\,dx,
\]
where $C_L>0$ and $\lambda\in(0,1)$ are some absolute constants.
\end{theorem}
The sharp maximal functions can be related as follows.
\begin{lemma}[{\cite[Proposition~2.3]{KL05}}]\label{le:relation}
If $\delta>0,\lambda\in(0,1)$, and $f\in L_{\rm loc}^\delta(\R)$,
then
\[
M_\lambda^\# f(x)\le (1/\lambda)^{1/\delta}f_\delta^\#(x)
\quad (x\in\R).
\]
\end{lemma}
The following estimate was proved in \cite[Theorem~2.1]{AP94}
for the case of Calder\'on-Zygmund singular integral operators
with standard kernels in the sense of Coifman and Meyer on $\R^n$.
It is well known that the Cauchy kernel is an archetypical example
of a standard kernel (see e.g. \cite[p.~99]{D01}).
\begin{theorem}[\'Alvarez-P\'erez]\label{th:Alvarez-Perez}
If $0<\delta<1$, then for every $f\in C^{\infty}_0(\R)$,
\[
(Sf)_\delta^\#(x)\le C_{\delta}Mf(x)
\quad (x\in\R)
\]
where $C_\delta>0$ is some constant depending only on $\delta$.
\end{theorem}
\subsection{Sufficient condition}\label{subsec:sufficiency}
The set of all bounded sublinear operators on a Banach function space $Y(\R)$
will be denoted by $\widetilde{\cB}(Y(\R))$ and its subset of all bounded
linear operators will be denoted by $\cB(Y(\R))$.
\begin{theorem}\label{th:conditional-result}
Let $X(\R)$ be a separable Banach function space and $w:\R\to[0,\infty]$ be a
weight such that  $w\in X_{\rm loc}(\R)$ and $1/w\in X_{\rm loc}'(\R)$. Suppose
the Hardy-Littlewood maximal operator $M$ is bounded on $X(\R,w)$ and on
$X'(\R,w^{-1})$. Assume that $0<\delta<1$ and $T$ is an operator such that
\begin{enumerate}
\item[{\rm(a)}]
$T$ is of weak type $(1,1)$;

\item[{\rm(b)}]
$T$ is bounded on some $L^p(\R)$ with $p\in(1,\infty)$;

\item[{\rm(c)}]
for each $f\in C_0^\infty(\R)$,
\[
(Tf)_\delta^\#(x)\le C_\delta Mf(x)\quad (x\in\R)
\]
where $C_\delta$ is a positive constant depending only on $\delta$.
\end{enumerate}
Then $T\in\cB(X(\R,w))$ and
\begin{equation}\label{eq:norm-estimate}
\|T\|_{\cB(X(\R,w))}\le(1/\lambda)^\delta C_L
\|M\|_{\widetilde{\cB}(X(\R,w))}\|M\|_{\cB(X'(\R,w^{-1}))}C_\delta,
\end{equation}
where $\lambda\in(0,1)$ and $C_L>0$ are the constants from Theorem~\ref{th:Lerner}.
\end{theorem}
\begin{proof}
The idea of the proof is borrowed from \cite[Theorem~2.7]{KL05}.
By Lemma~\ref{le:WBFS}, $X(\R,w)$ is a Banach function space whose
associate space is $X'(\R,w^{-1})$.
Let $f\in C_0^\infty(\R)$ and $g\in X'(\R,w^{-1})\subset L_{\rm loc}^1(\R)$.
Taking into account that $T$ is of weak type $(1,1)$, we see that $Tf$ satisfies
\eqref{eq:Lerner}. From Theorem~\ref{th:Lerner} we get that there
exist constants $\lambda\in(0,1)$ and $C_L>0$ independent of $f$ and $g$ such
that
\begin{equation}\label{eq:sufficiency-1}
\int_\R|(Tf)(x)g(x)|\,dx
\le
C_L\int_\R M_{\lambda}^\# (Tf)(x)Mg(x)\,dx.
\end{equation}
Since $T$ is bounded on some standard Lebesgue space $L^p(\R)$ for $1<p<\infty$
and $L^s(J)\subset L^r(J)$ whenever $0<r<s<\infty$ and $J$ is a finite interval,
we see that $Tf\in L^\delta_{\rm loc}(\R)$ for each $\delta\in(0,p]$.
From Lemma~\ref{le:relation} and hypothesis (c) it follows that
\begin{equation}\label{eq:sufficiency-2}
M_\lambda^\# (Tf)(x)\le
(1/\lambda)^{1/\delta}(Tf)_\delta^\#(x)
\le
(1/\lambda)^{1/\delta}
C_{\delta}Mf(x)
\quad (x\in\R)
\end{equation}
for some $\delta\in(0,1)$. Combining \eqref{eq:sufficiency-1} and \eqref{eq:sufficiency-2}
with H\"older's inequality (see Lemma~\ref{le:Hoelder}), we obtain
\begin{align}
\int_\R|(Tf)(x)g(x)|\,dx
&\le
C_1\int_\R Mf(x)Mg(x)\,dx
\nonumber
\\
&\le
C_1\|Mf\|_{X(\R,w)}\|Mg\|_{X'(\R,w^{-1})},
\label{eq:sufficiency-3}
\end{align}
where $C_1:= (1/\lambda)^{1/\delta} C_{\delta} C_L>0$ is independent of $f\in C_0^\infty(\R)$
and $g\in X'(\R,w^{-1})$. Taking into account that $M$ is bounded on $X(\R,w)$
and on $X'(\R,w^{-1})$, from \eqref{eq:sufficiency-3} and we get
\[
\int_\R|(Tf)(x)g(x)|\,dx \le C_2\|f\|_{X(\R,w)}\|g\|_{X'(\R,w^{-1})},
\]
where $C_2:=C_1\|M\|_{\widetilde{\cB}(X(\R,w))}\|M\|_{\widetilde{\cB}(X'(\R,w^{-1}))}$.
From this inequality and \eqref{eq:corollary-LS-1} we obtain
\begin{align*}
\|Tf\|_{X(\R,w)}
&=
\sup\left\{\int_\R|(Tf)(x)g(x)|\,dx\ :\
g\in X'(\R,w^{-1}),\ \|g\|_{X'(\R,w^{-1})}\le 1\right\}
\\
&\le
C_2\|f\|_{X(\R,w)}
\end{align*}
for all $f\in C_0^\infty(\R)$. Taking into account that $C_0^\infty(\R)$
is dense in $X(\R,w)$ in view of Lemma~\ref{le:density}(a), from the latter
inequality it follows that $T$ is bounded on $X(\R,w)$ and \eqref{eq:norm-estimate}
holds.
\end{proof}
\begin{remark}
The proof of this result without changes extends to the case of $\R^n$.
\end{remark}
\begin{theorem}\label{th:sufficiency}
Let $X(\R)$ be a Banach function space and $w:\R\to[0,\infty]$ be a weight
such that $w\in X_{\rm loc}(\R)$ and $1/w\in X_{\rm loc}'(\R)$. Suppose
the Hardy-Littlewood maximal operator $M$ is bounded on $X(\R,w)$ and on
$X'(\R,w^{-1})$.
\begin{enumerate}
\item[{\rm(a)}]
If the space $X(\R)$ is separable, then the Cauchy singular integral operator
$S$ is bounded on the space $X(\R,w)$ and $S^2=I$.

\item[{\rm(b)}]
If the space $X(\R)$ is reflexive, then the Cauchy singular integral operator
$S$ is bounded on the spaces $X(\R,w)$ and $X'(\R,w^{-1})$ and its adjoint
$S^*$ coincides with $S$ on the space $X'(\R,w^{-1})$.
\end{enumerate}
\end{theorem}
\begin{proof}
From Theorems~\ref{th:Kolmogorov-Riesz} and~\ref{th:Alvarez-Perez} it follows
that all hypotheses of Theorem~\ref{th:conditional-result} are fulfilled.
Hence, the operator $S$ is bounded on $X(\R,w)$.

Let now $\varphi\in X(\R,w)$. Then there exists a sequence $f_n\in L_0^\infty(\R)$
such that $f_n\to\varphi$ in $X(\R,w)$ as $n\to\infty$. From \eqref{eq:nice-relations-1} we get
$S^2f_n=f_n$ because $L_0^\infty(\R)\subset L^2(\R)$. Hence
\begin{align*}
\|S^2\varphi-\varphi\|_{X(\R,w)}
&\le
\|S^2\varphi-f_n\|_{X(\R,w)}+\|f_n-\varphi\|_{X(\R,w)}
\\
&=
\|S^2(\varphi-f_n)\|_{X(\R,w)}+\|\varphi-f_n\|_{X(\R,w)}
\\
&\le
(\|S^2\|_{\cB(X(\R,w))}+1)\|\varphi-f_n\|_{X(\R,w)}\to 0
\end{align*}
as $n\to\infty$. Thus $S^2\varphi=\varphi$. Part (a) is proved.

\medskip
(b) From \eqref{eq:nice-relations-2} it follows that
\[
\int_\R (Sf)(x)\overline{g(x)}\,dx = \int_\R f(x)\overline{(Sg)(x)}\,dx.
\]
for all $f,g\in L_0^\infty(\R)$. From this equality and
Lemmas~\ref{le:duality-reflexivity} and~\ref{le:density}(b)
it follows that $S$ is a self-adjoint
and densely defined operator on $X(\R,w)$ and $X'(\R,w^{-1})$.
By the standard argument (see \cite[Chap.~III, Section~5.5]{K95}),
one can show that $S=S^*\in\cB(X'(\R,w^{-1}))$ because
$S\in\cB(X(\R,w))$ by part (a).
\end{proof}
\subsection{Necessary condition}\label{subsec:necessity}
Let $X(\R)$ be a Banach function space and $X'(\R)$ be its associate space.
We say that a weight $w\colon \R\to[0,\infty]$ belongs to the class $A_X(\R)$ if
\[
\sup_{-\infty<a<b<\infty}
\frac{1}{b-a}\|w\chi_{(a,b)}\|_{X(\R)}\|w^{-1}\chi_{(a,b)}\|_{X'(\R)}<\infty.
\]
\begin{theorem}\label{th:bound-nec}
Let $X(\R)$ be a Banach function space and $w\colon \R\to[0,\infty]$ be a weight.
If the operator $S$ is bounded on the space $X(\R,w)$, then
\begin{enumerate}
\item[{\rm (a)}]
$w\in X_{\rm loc}(\R)$ and $1/w\in X_{\rm loc}'(\R)$;

\item[{\rm (b)}]
$X(\R,w)$ is a Banach function space;

\item[{\rm (c)}]
$w\in A_X(\R)$.
\end{enumerate}
\end{theorem}
\begin{proof}
(a) The idea of the proof is borrowed from \cite[Lemma~3.3]{K98}.
Let $E\subset\R$ be a measurable set of finite measure. Then there exist
$a,b\in\R$ such that $E\subset (a,b)=:J$. It is clear that
\[
w(x)\chi_E(x)\le w(x)\chi_J(x),
\quad
\chi_E(x)/w(x)\le \chi_J(x)/w(x)
\]
for almost all $x\in\R$. Then by Axiom (A2),
\[
\|w\chi_E\|_{X(\R)}\le \|w\chi_J\|_{X(\R)},
\quad
\|\chi_E/w\|_{X'(\R)}\le \|\chi_J/w\|_{X'(\R)}.
\]
Thus, it is sufficient to prove that $w\chi_J\in X(\R)$ and $\chi_J/w\in X'(\R)$.

Obviously, the operator $(Vf)(x)=\chi_J(x)xf(x)$ is bounded on $X(\R,w)$ and
\[
\big((SV-VS)f\big)(x)=\frac{1}{\pi i}\int_J f(y)\,dy
\]
for almost all $x\in\R$. Since the operator $SV-VS$ is bounded on $X(\R,w)$,
there exists a constant $C_1>0$ such that
\begin{equation}\label{eq:bound-nec-1}
\left\|\frac{1}{\pi i}\int_J f(y)dy\right\|_{X(\R,w)}
\le C_1\|f\|_{X(\R,w)}
\quad\mbox{for all}\quad f\in X(\R,w).
\end{equation}
On the other hand,
\begin{equation}\label{eq:bound-nec-2}
\left\|\frac{1}{\pi i}\int_J f(y)dy\right\|_{X(\R,w)}=
\frac{1}{\pi}\left|\int_Jf(y)dy\right|\,\|w\chi_J\|_{X(\R)}.
\end{equation}
Since $w(x)>0$ a.e. on $\R$, we have $\|w\chi_J\|_{X(\R)}>0$. Hence, from
\eqref{eq:bound-nec-1}--\eqref{eq:bound-nec-2} it follows that
\[
\left|\int_Jf(y)dy\right|\le\frac{C_1\pi}{\|w\chi_J\|_{X(\R)}}\|f\|_{X(\R,w)}.
\]
Therefore,
\[
\left|\int_\R f(y)w(y)\cdot\frac{\chi_J(y)}{w(y)}dy\right|
\le
\frac{C_1\pi}{\|w\chi_J\|_{X(\R)}}\|fw\|_{X(\R)}
\]
for all measurable functions $f$ such that $fw\in X(\R)$.
By Lemma~\ref{le:Landau}, we have $\chi_J/w\in X'(\R)$.

Let us show that there exists a function $g_0\in X(\R)$ such that
\begin{equation}\label{eq:bound-nec-3}
C_2:=\frac{1}{\pi}\left|\int_J\frac{g_0(y)}{w(y)}dy\right|>0.
\end{equation}
Assume the contrary. Then, taking into account Lemma~\ref{le:nice-subset}(a),
we obtain
\begin{equation}\label{eq:bound-nec-4}
\int_J\frac{g(y)}{w(y)}dy=0
\end{equation}
for all $g$ continuous on $\overline{J}$. By Axiom (A5), $(1/w)|_J\in L^1(J)$.
Without loss of generality, assume that $|J|=2\pi$. Let $\eta:[0,2\pi]\to\overline{J}$
be a homeomorphism such that $|\eta'(x)|=1$ for almost all $x\in [0,2\pi]$.
From \eqref{eq:bound-nec-4} we get
\begin{equation}\label{eq:bound-nec-5}
\int_0^{2\pi}\frac{\varphi(x)}{w(\eta(x))}dx=0
\quad\mbox{for all}\quad\varphi\in C[0,2\pi].
\end{equation}
Taking $\varphi(x)=e^{inx}$ with $n\in\Z$, we see from \eqref{eq:bound-nec-5}
that all Fourier coefficients of $1/(w\circ\eta)$ vanish. This implies that
$1/w(\eta(x))=0$ for almost all $x\in [0,2\pi]$. Consequently, $w(y)=\infty$
almost everywhere on $J$. This contradicts the assumption that $w$ is a weight.
Thus, $C_2>0$.

Clearly, $f_0 =g_0/w\in X(\Gamma,w)$. Then from \eqref{eq:bound-nec-1}--\eqref{eq:bound-nec-3}
it follows that
\[
\|w\chi_J\|_{X(\R)}\le\frac{C_1\pi}{C_2}\|f_0\|_{X(\R,w)},
\]
that is, $w\chi_J\in X(\R)$. Part (a) is proved.

\medskip
Part (b) follows from part (a) and Lemma~\ref{le:WBFS}(b).

\medskip
(c) The idea of the proof is borrowed from \cite[Theorem~3.2]{K98}. By part(b),
$X(\R,w)$ is a Banach function space.

Let $Q$ be an arbitrary interval and $Q_1,Q_2$ be its two halves. Take a
function $f\ge 0$ supported in $Q_1$. Then for $\tau\in Q_1$ and $x\in Q_2$ we
have $|\tau-x|\le |Q|$. Therefore,
\begin{align*}
|(Sf)(x)| &=
\frac{1}{\pi}\left|\int_{Q_1}\frac{f(\tau)}{\tau-x}d\tau\right|
=
\frac{1}{\pi}\int_{Q_1}\frac{f(\tau)}{|\tau-x|}d\tau
\\
&\ge \frac{1}{\pi|Q|}\int_{Q_1}f(\tau)\,d\tau
=
\frac{1}{2\pi|Q_1|}\int_{Q_1}f(\tau)\,d\tau.
\end{align*}
Thus,
\[
|(Sf)(x)|\chi_{Q_2}(x)\ge
\frac{1}{2\pi|Q_1|}\left(\int_{Q_1}f(\tau)\,d\tau\right)\chi_{Q_2}(x)
\quad(x\in\R).
\]
Then, by Axioms (A1) and (A2),
\begin{equation}\label{eq:bound-nec-6}
\|Sf\|_{X(\R,w)}
\ge
\|(Sf)\chi_{Q_2}\|_{X(\R,w)}
\ge
\frac{1}{2\pi|Q_1|}\left(\int_{Q_1}f(\tau)\,d\tau\right)\|\chi_{Q_2}\|_{X(\R,w)}.
\end{equation}
On the other hand, since $S$ is bounded on $X(\R,w)$, we get
\begin{equation}\label{eq:bound-nec-7}
\|Sf\|_{X(\R,w)}
\le
\|S\|_{\cB(X(\R,w))}\|f\|_{X(\R,w)}
=
\|S\|_{\cB(X(\R,w))}\|f\chi_{Q_1}\|_{X(\R,w)}.
\end{equation}
Combining \eqref{eq:bound-nec-6} and \eqref{eq:bound-nec-7}, we arrive at
\begin{equation}\label{eq:bound-nec-8}
\frac{1}{|Q_1|}\left(\int_{Q_1}f(\tau)\,d\tau\right)\|w\chi_{Q_2}\|_{X(\R)}
\le
2\pi\|S\|_{\cB(X(\R,w))}\|f\chi_{Q_1}\|_{X(\R,w)}.
\end{equation}
Taking $f=\chi_{Q_1}$, from \eqref{eq:bound-nec-8} we get
\[
\|w\chi_{Q_2}\|_{X(\R)}
\le
2\pi\|S\|_{\cB(X(\R,w))}\|w\chi_{Q_1}\|_{X(\R)}.
\]
Analogously one can obtain
\begin{equation}\label{eq:bound-nec-9}
\|w\chi_{Q_1}\|_{X(\R)}
\le
2\pi\|S\|_{\cB(X(\R,w))}\|w\chi_{Q_2}\|_{X(\R)}.
\end{equation}
From \eqref{eq:bound-nec-8} and \eqref{eq:bound-nec-9} it follows that
\begin{equation}\label{eq:bound-nec-10}
\frac{1}{|Q_1|}\left(\int_{Q_1}f(\tau)\,d\tau\right)\|w\chi_{Q_1}\|_{X(\R)}
\le
C\|f\chi_{Q_1}\|_{X(\R,w)},
\end{equation}
where $C:=\left(2\pi\|S\|_{\cB(X(\R,w))}\right)^2$. Let
\[
Y:=\left\{g\in X(\R,w): \|g\|_{X(\R,w)}\le 1\right\}.
\]
If $g\in Y$, then $|g|\chi_{Q_1}\ge 0$ is supported in $Q_1$. Then from
\eqref{eq:bound-nec-10} we obtain
\begin{equation}\label{eq:bound-nec-11}
\|w\chi_{Q_1}\|_{X(\R)}\int_\R|g(\tau)|\chi_{Q_1}(\tau)\,d\tau
\le
C|Q_1|
\end{equation}
for all $g\in Y$. From  \eqref{eq:corollary-LS-2} we get
\begin{equation}\label{eq:bound-nec-12}
\|w^{-1}\chi_{Q_1}\|_{X'(\R)}
=
\|\chi_{Q_1}\|_{X'(\R,w^{-1})}
=
\sup_{g\in Y}\int_\R|g(\tau)|\chi_{Q_1}(\tau)\,d\tau.
\end{equation}
From \eqref{eq:bound-nec-11} and \eqref{eq:bound-nec-12} it follows that
\[
\|w\chi_{Q_1}\|_{X(\R)}\|w^{-1}\chi_{Q_1}\|_{X'(\R)}\le C|Q_1|.
\]
Since $Q_1\subset\R$ is an arbitrary interval, we conclude that $w\in A_X(\R)$.
\end{proof}
\subsection{The case of weighted variable Lebesgue spaces}\label{subsec:WVLS}
We start this subsection with the following well-known fact.
\begin{theorem}[{\cite[Theorems~3.2.13 and 3.4.7]{DHHR11}}]
\label{th:VLS-BFS}
Let $p:\R\to[1,\infty]$ be a measurable a.e. finite function satisfying
\eqref{eq:exponents}. Then $L^{p(\cdot)}(\R)$ is a separable and reflexive
Banach function space whose associate space is isomorphic to $L^{p'(\cdot)}(\R)$.
\end{theorem}
Now we are in a position to give a proof of Theorem~\ref{th:main}.
\begin{proof}[Proof of Theorem~\ref{th:main}]
\textit{Necessity.}
Theorem~\ref{th:VLS-BFS} immediately implies that if $p$ satisfies \eqref{eq:exponents},
then $L^{p(\cdot)}(\R)$ is a Banach function space and
\[
\cA_{p(\cdot)}(\R)=A_{L^{p(\cdot)}}(\R).
\]
From Theorem~\ref{th:bound-nec} it follows that that if $S$ is bounded on
the space $L^{p(\cdot)}(\R,w)$, then $w\in\cA_{p(\cdot)}(\R)$. The necessity
portion is proved.

\textit{Sufficiency.} From Theorem~\ref{th:VLS-BFS} we know that $L^{p(\cdot)}(\R)$
is a separable and reflexive Banach function space. If $w\in\cA_{p(\cdot)}(\R)$,
then $w\in L^{p(\cdot)}_{\rm loc}(\R)$, $1/w\in L^{p'(\cdot)}_{\rm loc}(\R)$, and
$1/w\in\cA_{p'(\cdot)}(\R)$. Further, it is easy to see that $p$ is globally log-H\"older
continuous if and only if so is $p'$. Hence, by Theorem~\ref{th:CDH}, the Hardy-Littlewood
maximal function is bounded on $L^{p(\cdot)}(\R,w)$ and on $L^{p'(\cdot)}(\R,w^{-1})$.
Applying Theorem~\ref{th:sufficiency}(a), we see that the operator $S$ is bounded
on $L^{p(\cdot)}(\R,w)$. This finishes the proof of Theorem~\ref{th:main}.
\end{proof}
Theorem~\ref{th:nice-relations} follows immediately from
Theorems~\ref{th:CDH}, \ref{th:sufficiency}, and \ref{th:VLS-BFS}.

\end{document}